\documentclass{article}
\usepackage[english]{babel}
\usepackage[cp1251]{inputenc}
\usepackage{amsmath}
\usepackage{amssymb}
\usepackage{amsfonts}
\usepackage{amsthm}
\usepackage{graphicx}

\def\thtext#1{
  \catcode`@=11
  \gdef\@thmcountersep{. #1}
  \catcode`@=12
}

\def\threst{
  \catcode`@=11
  \gdef\@thmcountersep{.}
  \catcode`@=12
}

 \catcode`@=11
 \def\.{.\spacefactor\@m}
 \catcode`@=12

\theoremstyle{plain}
\newtheorem{thm}{Theorem}[section]
\newtheorem{prop}{Proposition}[section]
\newtheorem{cor}[prop]{Corollary}

\theoremstyle{definition}
\newtheorem{rk}[prop]{Remark}

\newcommand{\cM}{\mathcal{M}}

\newcommand{\N}{\mathbb{N}}
\newcommand{\R}{\mathbb{R}}

\newcommand{\dl}{\delta}
\newcommand{\D}{\Delta}

\newcommand{\diam}{\operatorname{diam}}

\newcommand{\gen}{{\operatorname{gen}}}

\newcommand{\mst}{{\operatorname{mst}}}

\newcommand{\mf}{{\operatorname{mf}}}

\newcommand{\smt}{{\operatorname{smt}}}
\newcommand{\SMT}{{\operatorname{SMT}}}
\newcommand{\sr}{{\operatorname{sr}}}
\newcommand{\ssr}{{\operatorname{ssr}}}
\newcommand{\sgr}{{\operatorname{sgr}}}

\newcommand{\0}{\emptyset}
\renewcommand{\:}{\colon}

\renewcommand{\sp}{\supset}
\renewcommand{\ss}{\subset}

\begin{document}
\title{Steiner Ratio and Steiner--Gromov Ratio of Gromov--Hausdorff Space}
\author{Alexander O.Ivanov, \and Alexey A.Tuzhilin}
\date{}
\maketitle

\begin{abstract}
In the present paper we investigate the metric space $\cM$ consisting of isometry classes of compact metric spaces, endowed with the Gromov--Hausdorff metric. We show that for any finite subset $M$ from a sufficiently small neighborhood of a generic finite metric space, providing $M$ consists of finite metric spaces with the same number of points, each Steiner minimal tree in $\cM$ connecting $M$ is a minimal filling for $M$. As a consequence, we prove that the both Steiner ratio and Gromov--Steiner ratio of $\cM$ are equal to $1/2$.
\end{abstract}

\section{Introduction}
\markright{\thesection.~Introduction}
By $\cM$ we denote the set of all isometry classes of compact metric spaces, endowed with the Gromov--Hausdorff metric. In~\cite{IvaIliadisTuzIsom} we defined generic spaces as finite metric spaces such that all their nonzero distances are pairwise distinct, and all their triangle inequalities hold strictly. Such spaces form an everywhere dense, but not open subset of $\cM$. We have shown that any sufficiently small neighborhood of a generic $n$-point compact metric space, being restricted to the family of all $n$-point metric spaces, is isometric to an open set in some $\R^k_{\infty}$, i.e., in $\R^k$ with the norm $\bigl\|(x_1,\ldots,x_k)\bigr\|=\max\bigl\{|x_i|\bigr\}$. By means of these isometries, together with isometric Kuratowski embedding of a $k$-point metric space into $\R^k_{\infty}$, we have shown that each finite metric space can be isometrically embedded into $\cM$. In the present paper we use these results to show that for the boundaries which belong to sufficiently small neighborhoods of the generic spaces, and consist of finite metric spaces with the same number of points, all Steiner minimal trees are minimal fillings, see~\cite{ITMinFil} for definitions and details. Generalizing~\cite{IvaIliadisTuzIsom}, we get, as a consequence, that all minimal fillings of finite metric spaces can be realized in $\cM$ as Steiner minimal trees. The latter immediately implies that the both Steiner ratio and Gromov--Steiner ratio of the space $\cM$ are equal to $1/2$.

\section{Preliminaries}
\markright{\thesection.~Preliminaries}
Let $X$ be an arbitrary metric space.  By $|xy|$ we denote the distance between points $x$ and $y$ in $X$. For each point $x\in X$ and a real number $r>0$ by $U_r(x)$ we denote  an open ball of radius $r$  centered at $x$; for every nonempty $A\ss X$ and a real number $r>0$ put $U_r(A)=\cup_{a\in A}U_r(a)$.

For nonempty $A,\,B\ss X$ put
$$
d_H(A,B)=\inf\bigl\{r>0:A\ss U_r(B)\&B\ss U_r(A)\bigr\}.
$$
This value is called the \emph{Hausdorff distance between $A$ and $B$}. It is well-known~\cite{BurBurIva} that the Hausdorff distance is a metric on the set of all closed bounded subsets of $X$.

Let $X$ and $Y$ be metric spaces. A triple $(X',Y',Z)$ consisting of a metric space $Z$ and two its subsets $X'$ and $Y'$ isometrical to $X$ and $Y$, respectively, is called a \emph{realization of the pair $(X,Y)$}. The \emph{Gromov--Hausdorff distance $d_{GH}(X,Y)$ between $X$ and $Y$} is the greatest lower bound of the real numbers $r$ such that there exists a realization $(X',Y',Z)$ of the pair $(X,Y)$ that satisfies $d_H(X',Y')\le r$. It is well-known~\cite{BurBurIva} that the $d_{GH}$ is a metric on the set $\cM$ of isometry classes of compact metric spaces.

Let $G=(V,E)$ be an arbitrary graph with vertex set $V$ and edge set $E$. We say that the graph $G$ is \emph{defined on a metric space $X$} if $V\ss X$. For any such a graph the \emph{length $|e|$} of any its \emph{edge $e=vw$} is defined as the distance $|vw|$ between the ends $v$ and $w$ of the edge; the latter generates the \emph{length $|G|$} of the \emph{graph $G$} as the sum of the lengths of all its edges.

If $M\ss X$ is an arbitrary finite nonempty subset, and $G=(V,E)$ is a graph on $X$, then we say that $G$ \emph{connects $M$} if $M\ss V$. The greatest lower bound of the length of connected graphs connecting $M\ss X$ is called the \emph{length of Steiner minimal tree on $M$}, or the \emph{length of the shortest tree on $M$}, and it is denoted by $\smt(M,X)$. Each connected graph $G$ that connects $M$ and such that $|G|=\smt(M,X)$ is a tree which is called a \emph{Steiner minimal tree on $M$}, or a \emph{shortest tree on $M$}. By $\SMT(M,X)$ we denote the set of all Steiner minimal trees on $M$. Notice that the set $\SMT(M,X)$ may be empty, and that the both $\SMT(M,X)$ and $\smt(M,X)$ depend not only on the distances between points in $M$, but also on geometry of the ambient space $X$, e.g., isometrical $M$ belonging to distinct metric spaces $X$ may be connected by shortest trees of nonequal lengths. For details of the shortest trees theory see~\cite{IvaTuz} and~\cite{HwRiW}.

\begin{prop}[\cite{IvaTuz}]\label{prop:existance}
If $X$ is a complete Riemannian manifold, or finite-dimensional normed vector space, then for every finite nonempty set $M\ss X$ it holds $\SMT(M,X)\ne\0$.
\end{prop}

Let $M$ be a finite metric space. Define $\mst(M)$ as the length of the shortest tree of the form $(M,E)$. This value is called the \emph{length of minimal spanning tree on $M$}; a tree $G=(M,E)$ with $|G|=\mst(M)$ is called a \emph{minimal spanning tree on $M$}. Notice that for any $M$ there always exists a minimal spanning tree on it.

Now we fix a metric space $M$, consider isometrical embeddings of $M$ into various metric spaces $X$, and for every such embedding calculate the length of Steiner minimal tree on the image of $M$. The ``shortest'' length of these Steiner minimal trees over all embeddings we call the \emph{length of minimal filling of $M$}, and this value is denoted by $\mf(M)$. To overcome the problems like Cantor's paradox, let us give a more accurate definition. The number $\mf(M)$ is the greatest lower bound of the real numbers $r$ such that there exist a metric space $X$ and an isometric embedding $\mu\:M\to X$ with $\smt\bigl(\mu(M),X\bigr)\le r$. If $\SMT\bigl(\mu(M),X\bigr)\ne\0$ and $\smt\bigl(\mu(M),X\bigr)=\mf(M)$, then each $G\in\SMT\bigl(\mu(M),X\bigr)$ is called a \emph{minimal filling for $M$}. In~\cite{ITMinFil} we have given an equivalent definition of minimal filling, which does not use ambient spaces $X$, and we have shown that for each finite metric space $M$ a minimal filling always exists.

The next result can be easily obtained from~\cite{ITMinFil}.

\begin{prop}[\cite{ITMinFil}]
For any finite metric space $M$, all its minimal fillings are just the graphs $G=(V,E)$ on a metric space $V$ extending $M$, such that $|G|=\mst(V)=\mf(M)$.
\end{prop}

Let $M$ be an arbitrary subset of a metric space $X$, and suppose that $M$ consists of at least two points. Then $\smt(M,X)>0$, $\mf(M)>0$, and $\mst(M)>0$. The next three values characterize the relations between the three types of trees defined above:
\begin{itemize}
\item the \emph{Steiner ratio\/} $\sr(M,X)=\smt(M,X)/\mst(M)$;
\item the \emph{Steiner--Gromov ratio\/} $\sgr(M)=\mf(M)/\mst(M)$;
\item the \emph{Steiner subratio\/} $\ssr(M,X)=\mf(M)/\smt(M,X)$.
\end{itemize}
If we take the greatest lower bounds of the ratios defined above over all such $M\ss X$, then we get the corresponding ratios for the ambient space $X$. Notice that the calculation of these ratios is a very difficult problem. For details on one-dimensional minimal fillings see~\cite{ITMinFil}.

By $\R^n_\infty$ we denote  the space $\R^n$ with the norm
$$
\bigl\|(x_1,\ldots,x_n)\bigr\|=\max\bigl\{|x_i|\bigr\}.
$$
Recall that every finite metric space can be isometrically embedded into some $\R^n_\infty$ by means of \emph{Kuratowski embedding}, that definition we give in the next Proposition.

\begin{prop}[\cite{Kuratowski}]\label{prop:Kuratowski}
The mapping from a metric space $X=\{x_1,\ldots,x_n\}$ into $\R^n_\infty$ defined as
$$
\chi\:x_i\mapsto\bigl(|x_1x_i|,\ldots,|x_nx_i|\bigr)
$$
is an isometrical embedding.
\end{prop}

\begin{prop}[Z.N.Ovsyannikov, \cite{Ovs}, \cite{IvaTuzIrreducible}]\label{prop:Ovsyannikov}
Each shortest tree in $\R^n_\infty$ is a minimal filling for its boundary.
\end{prop}

By $\cM_{[n]}$ we denote the subset of $\cM$ consisting of all metric spaces each of which consists of exactly $n$ points. Put $\cM_n=\cup_{1\le k\le n}\cM_{[k]}$.

The next Proposition states the existence of shortest trees in $\cM$ for boundaries consisting of finite metric spaces.

\begin{prop}[\cite{IvaNikolaevaTuzSteiner}]\label{prop:SMTexistance}
For every $M\ss\cM_n$ it holds
$$
\SMT(M,\cM)\ne\0.
$$
\end{prop}

For every $n\in\N$ we define a mapping $\nu\:\cM_{[n]}\to\R^{n(n-1)/2}_{\infty}$ as follows. For $X=\{x_1,\ldots,x_n\}\in\cM_{[n]}$ consider all distances $|x_ix_j|$ between distinct points, sort these values in the ascending order, and by $\nu(X)$ we denote  the resulting vector from $\R^{n(n-1)/2}$ divided by $2$.

By a \emph{generic space\/} we mean each finite metric space such that all its nonzero distances are pairwise distinct, and all its triangle inequalities hold strictly.  By $\cM^\gen$ we denote the family of all generic spaces. Clearly that $\cM^\gen$ is everywhere dense in $\cM$. Also, we put $\cM_n^\gen=\cM_n\cap\cM^\gen$ and $\cM_{[n]}^\gen=\cM_{[n]}\cap\cM^\gen$.

Let $X=\{x_1,\ldots,x_n\}\in\cM$, $n\ge3$. Define $\dl(X)$ as the minimum of the following two values:
\begin{flalign*}
\indent&\min\bigl\{|x_ix_j|+|x_jx_k|-|x_ix_k|:\#\{i,j,k\}=3\bigr\},&\\
\indent&\min\Bigl\{\bigl||x_ix_j|-|x_px_q|\bigr|:\#\{i,j,p,q\}\ge3\Bigr\}.&
\end{flalign*}
For $n=2$ put $\dl(X)=|x_1x_2|$.

\begin{prop}[\cite{IvaIliadisTuzIsom}]\label{prop:isometry}
Let $X=\{x_1,\ldots,x_n\}\in\cM^\gen$. Put $\dl=\dl(X)/6$, $U=U_\dl(X)\ss\cM_n$, $N=n(n-1)/2$, $W=U_\dl\bigl(\nu(X)\bigr)\ss\R^N_\infty$. Then the mapping $\nu|_U\:U\to W$ is an isometry.
\end{prop}

\begin{rk}
The neighborhood $U=U_\dl(X)\ss\cM_n$ does not contain the spaces with less than $n$ points, thus the mapping $\nu$ is well-defined on $U$.
\end{rk}

\begin{prop}[\cite{IvaIliadisTuzIsom}]\label{prop:isometry-embedding}
Let $X$ be an arbitrary finite metric space consisting of $n$ points. By $k$ we denote the least integer such that $n\le k(k-1)/2$. Then $X$ can be isometrically embedded into $\cM$. Moreover, the embedding can be chosen in such a way that the image of $X$ is contained in $\cM_{[k]}$.
\end{prop}

\section{Shortest Trees in Small Neighborhoods of Generic Spaces}
\markright{\thesection.~Shortest Trees in Small Neighborhoods of Generic Spaces}

The next theorem reduces the Steiner Problem in the Gromov--Hausdorff space for boundaries lying in sufficiently small neighborhoods of generic spaces, to the investigation of minimal fillings for finite metric spaces.

\begin{thm}
Let $X\in\cM_{[n]}^\gen$. Then for every $N\in\N$ there exists a neighborhood $U_r(X)\ss\cM_n$ such that for any set $M\ss U_r(X)$ consisting of at most $N$ points, each its Steiner minimal tree, that exists by Proposition~$\ref{prop:SMTexistance}$, is a minimal filling.
\end{thm}

\begin{proof}
Let $\nu\:\cM_{[n]}\to\R^{n(n-1)/2}_{\infty}$ be the mapping  defined above. Put $Y=\nu(X)$ and $R=\R^{n(n-1)/2}_{\infty}$. By Proposition~\ref{prop:isometry}, there exists $\dl>0$ such that $\nu$ is an isometry between $U=U_\dl(X)\ss\cM_n$ and $W=U_\dl(Y)\ss R$.

Put $r=\dl/\big(1+2(N-1)\big)$. Let $A\ss U_r(Y)$, then the distance between any pair of points from $A$ is less than $2r$. Suppose also that $A$ consists of at most $N$ points, then the length of any tree with vertex set $A$ is less than $2r(N-1)$, hence, $\smt(A,R)<2r(N-1)$.

Let $G=(V,E)\in\SMT(A,R)$. Then $\diam V\le|G|<2r(N-1)$, where $\diam V$ stands for the diameter of the set $V$. Since $A\ss V$, then the distance from a point that belongs to $V$, say from $p\in A$, to the set $Y$ is less than $r$, therefore, for any $v\in V$ we have
$$
|Yv|\le|Yp|+|pv|\le|Yp|+\diam V<r+2r(N-1)=\dl.
$$
Thus, we have shown that for each $G=(V,E)\in\SMT(A,R)$ it holds $V\ss W$.

Now, choose an arbitrary $M\ss U_r(X)$ and put $A=\nu(M)$. Then $A\ss U_r(Y)$. As it is proved above, for any $G=(V,E)\in\SMT(A,R)$ it holds $V\ss U_\dl(Y)$, therefore, $V'=\nu^{-1}(V)$ is isometric to $V$, and $V'\sp M$.

Let $G'=(V',E')$ be the unique graph such that $\nu^{-1}\:V\to V'$ is an isomorphism between $G$ and $G'$. Then $|G'|=|G|$, and $G'$ is a minimal filling for $M$, because $G$ is a minimal filling for $A$ by Proposition~\ref{prop:Ovsyannikov}. Since the lengths of all Steiner minimal trees on the same boundary are equal to each other, then all shortest trees on $M\ss\cM$ are minimal fillings for $M$.
\end{proof}

The next result shows that each minimal filling can be realized as a shortest tree in the Gromov--Hausdorff space.

\begin{thm}\label{thm:MFRealization}
Let $X$ be an arbitrary finite metric space, and suppose that $G=(V,E)$ is some minimal filling for $X$. Then there exist $N\in\N$ and isometric embedding $f\:X\to\cM_N$ such that
\begin{enumerate}
\item\label{thm:MFRealization:1} the mapping $f$ can be extended to an isometric embedding $F\:V\to\cM_N$,
\item\label{thm:MFRealization:2} the graph $F(G):=\bigl(F(V),F(E)\bigr)$ is a Steiner minimal tree in $\cM$ connecting $f(X)$, i.e.,  $F(G)\in\SMT\bigl(f(X),\cM\bigr)$, and
\item\label{thm:MFRealization:3} each Steiner minimal tree in $\cM$ connecting $f(X)$ is a minimal filling.
\end{enumerate}
\end{thm}

\begin{proof}
To verify item~(\ref{thm:MFRealization:1}) it suffices to use Proposition~\ref{prop:isometry-embedding} and to construct an isometric embedding $F$ of $V$ into $\cM$, and then to define $f$ as the restriction of $F$ onto $X$. Item~(\ref{thm:MFRealization:2}) follows from the fact that each connected graph connecting a finite subset of a metric space is not shorter than a minimal filling for this subset. To prove item~(\ref{thm:MFRealization:3}) it suffices to notice that all Steiner minimal trees on the same boundary have the same length.
\end{proof}

Theorem~\ref{thm:MFRealization} and~\cite{ITMinFil} imply the following result.

\begin{cor}
We have $\sr(\cM)=1/2$ and $\sgr(\cM)=1/2$.
\end{cor}

\begin{proof}
Indeed, if $\D_n$ is an $n$-point metric space with distances $1$ between any pair of its distinct points, then $\mst(\D_n)=n-1$. By~\cite{ITMinFil}, a minimal filling of $\D_n$ is a star-like graph $\bigl(\D_n\cup\{v\},E\bigr)$ with edges $vx_i\in E$, $x_i\in\D_n$, of the same length  $1/2$. Thus, $\mf(\D_n)=n/2$. Due to Theorem~\ref{thm:MFRealization}, embed $\D_n$ isometrically into $\cM$ in such a way that Steiner minimal trees on the image $M$ of $\D_n$ are minimal fillings. Then $\smt(M,\cM)=\mf(M)=\mf(\D_n)=n/2$ and $\mst(M)=n-1$, hence, $\sr(M,\cM)=\frac{n/2}{n-1}=\sgr(M)$ that tends to $1/2$ as  $n$ tends to infinity.
\end{proof}

\begin{rk}
In~\cite{IvaTuzIrreducible} it was shown that $\ssr(\cM)<0.857$. It would be interesting to get the exact value of $\ssr(\cM)$.
\end{rk}

\end{document}